\documentclass[11pt]{article}

\usepackage[T1]{fontenc}
\usepackage[utf8]{inputenc}
\usepackage{lmodern}
\usepackage{amsmath,amssymb,amsthm,mathtools}
\usepackage{booktabs,tabularx,array}
\usepackage{graphicx}
\usepackage{xcolor}
\usepackage{subcaption}
\usepackage{float}
\usepackage{enumitem}
\usepackage{microtype}
\usepackage[a4paper,margin=0.86in]{geometry}
\usepackage[colorlinks=true,linkcolor=blue!55!black,
            citecolor=blue!55!black,urlcolor=blue!55!black]{hyperref}

\graphicspath{{figures/}}
\setlist[itemize]{leftmargin=1.5em,itemsep=0.15em,topsep=0.25em}
\setlist[enumerate]{leftmargin=1.6em,itemsep=0.15em,topsep=0.25em}

\newtheorem{theorem}{Theorem}
\newtheorem{proposition}[theorem]{Proposition}
\newtheorem{lemma}[theorem]{Lemma}
\newtheorem{corollary}[theorem]{Corollary}
\theoremstyle{remark}
\newtheorem{remark}[theorem]{Remark}

\newcommand{\R}{\mathbb{R}}
\newcommand{\norm}[1]{\left\lVert #1\right\rVert}
\newcommand{\inner}[2]{\left\langle #1,#2\right\rangle}
\newcommand{\SAMa}{\mathrm{SAM}_{\alpha}}

\newcommand{\fmin}{f_{\inf}}

\title{\bf Stationarity Floors and Vanishing Perturbations\\
in Sharpness-Aware Minimization}

\author{
Samir Adly\thanks{Laboratoire XLIM, Universit\'e de Limoges,
123 Avenue Albert Thomas, 87060 Limoges CEDEX, France.
Email: \texttt{samir.adly@unilim.fr}}
\qquad
Ba Khiet Le\thanks{Analytical and Algebraic Methods in Optimization
Research Group, Faculty of Mathematics and Statistics,
Ton Duc Thang University, Ho Chi Minh City, Vietnam.
Email: \texttt{lebakhiet@tdtu.edu.vn}}
}

\date{}

\begin{document}
\maketitle

\begin{abstract}
We study a deterministic family of sharpness-aware minimization methods for
smooth nonconvex functions.  The perturbation is
$$
  y_k=x_k+\rho\,
  \frac{\nabla f(x_k)}{\norm{\nabla f(x_k)}^\alpha},
  \qquad 0\leq\alpha\leq 1,
$$
so that its effective radius is
$\rho\norm{\nabla f(x_k)}^{1-\alpha}$.
For $0<\alpha\leq1$, we give an explicit complexity bound above the
stationarity level $(L\rho)^{1/\alpha}$.  A one-dimensional quadratic
example reaches this level exactly, showing that the bound describes a real
limitation of the constant-parameter rule.  The unnormalized case
$\alpha=0$ is treated separately and requires $L\rho<1$.  We then
introduce a clipped rule which agrees with the constant-$\rho$ rule away
from stationary points and becomes proportional to the gradient near them.
The clipped method has
$\norm{\nabla f(x_k)}\to0$ and the usual $O(T^{-1/2})$ stationarity
bound.  Numerical tests on a quadratic function, the Rosenbrock function,
and a five-dimensional nonconvex function illustrate the stationarity floor
of the unclipped rule and the effect of clipping.
\end{abstract}


\tableofcontents
\section{Introduction}

We consider the unconstrained problem
\begin{equation}\label{eq:problem}
  \min_{x\in\R^n} f(x),
\end{equation}
where $f$ is differentiable, bounded from below, and has a Lipschitz
continuous gradient.  Sharpness-Aware Minimization (SAM), introduced by
Foret et al.~\cite{Foret2021}, uses the iteration
\begin{equation}\label{eq:sam}
  y_k=x_k+\rho\frac{\nabla f(x_k)}{\norm{\nabla f(x_k)}},
  \qquad
  x_{k+1}=x_k-h\nabla f(y_k).
\end{equation}
The perturbation in \eqref{eq:sam} has the fixed length $\rho$.  It is
motivated by a first-order approximation of the robust objective
$$
  x\longmapsto\max_{\norm{u}\leq\rho}f(x+u).
$$
SAM has been successful in the training of neural networks, and several
works have studied its optimization and generalization properties
\cite{Andriushchenko2022,Dai2023,Si2023,Khanh2024,Oikonomou2025}.

The meaning of sharpness needs some care.  Curvature measures, robust loss
increments, and Hessian-based quantities are not equivalent.  They may also
depend on the parametrization of a neural network
\cite{Dinh2017,Tahmasebi2024}.  In this paper we do not claim that a
deterministic SAM trajectory must select a flatter minimizer.  Our purpose is
to study how the perturbation scale affects stationarity.
This point is relevant because a constant-radius normalized perturbation may
prevent convergence to a stationary point.

The closest results already establish several parts of this picture.
Si and Yun~\cite{Si2023} analyze normalized SAM with a constant radius.
For deterministic smooth nonconvex functions, their bounds contain an
additive term of order $L^2\rho^2$, and their examples show that such a
residual term cannot in general be removed.  Dai, Ahn, and
Sra~\cite{Dai2023} explain how normalization can stabilize the dynamics and
can maintain motion along a manifold of minimizers; their purpose is to
compare normalized SAM with USAM {(Unnormalized Sharpness-Aware Minimization)} rather than to quantify the floor for the
family considered below.  Khanh et al.~\cite{Khanh2024} give a broad
inexact-gradient analysis.  For normalized variants in the nonconvex case,
their exact stationarity results use diminishing perturbation radii; for
USAM, they obtain convergence by controlling a relative gradient error.
Oikonomou and Loizou~\cite{Oikonomou2025} unify SAM and USAM through a
convex combination of their perturbations and study deterministic and
stochastic convergence under several sampling and stepsize choices. 

 In the present paper,  we consider the family
\begin{equation}\label{eq:samalpha-intro}
  y_k=x_k+\rho
  \frac{\nabla f(x_k)}{\norm{\nabla f(x_k)}^\alpha},
  \qquad
  x_{k+1}=x_k-h\nabla f(y_k),
  \qquad 0\leq\alpha\leq1.
\end{equation}
The effective perturbation radius is
\begin{equation}\label{eq:radius-intro}
  r_k=\norm{y_k-x_k}
      =\rho\norm{\nabla f(x_k)}^{1-\alpha}.
\end{equation}
We call \eqref{eq:samalpha-intro} the \emph{constant-$\rho$
$\SAMa$ rule}, or the \emph{unclipped $\SAMa$ rule}.  Here
``constant'' refers to the parameter $\rho$, not to the effective radius
$r_k$, which is constant only when $\alpha=1$.
The choice $\alpha=1$ is normalized SAM, while $\alpha=0$ is the
unnormalized method considered, for example, in
\cite{Andriushchenko2022,Khanh2024}.  Intermediate values reduce the
perturbation when the gradient becomes small, but, as shown below, this
reduction is not sufficient in general to obtain exact stationarity.

Within this  setting, the contributions are as follows.
\begin{itemize}
  \item For $0<\alpha\leq1$, we derive an explicit finite-time estimate
  above the stationarity floor
  $$
      \varepsilon_\alpha=(L\rho)^{1/\alpha}.
  $$
  The dependence on $\alpha$ and on all constants is kept visible.  This
  extends the constant-radius obstruction for normalized SAM to the
  exponent family \eqref{eq:samalpha-intro}.

  \item On the quadratic $f(x)=Lx^2/2$ with $h=1/L$, the gradient norms
  converge exactly to $(L\rho)^{1/\alpha}$.  Hence this quantity cannot be
  removed from a general smooth analysis.  The same example also gives the
  necessary stability condition $L\rho<1$ for the unnormalized case.

  \item We introduce a clipped radius which directly enforces a relative
  gradient-error condition.  It keeps the constant-$\rho$ $\SAMa$
  perturbation when the gradient is not small and makes the perturbation
  proportional to the gradient near stationarity.  We prove convergence of
  the whole gradient sequence to zero and an $O(T^{-1/2})$ bound.  The
  relative-error principle is not new; the contribution is the clipping
  rule that applies it to the present $\SAMa$ family without changing the
  update outside the final stationarity region.

  \item The numerical comparisons report gradient evaluations, not only
  iterations.  This accounts for the second gradient evaluation required by
  every SAM-type update.
\end{itemize}

The paper is organized as follows. First we recall some basic notions and definitions in Section \ref{sec2}. The \texorpdfstring{constant-$\rho$}{constant-rho} rule and its
stationarity floor is considered in Section \ref{sec3}. In Section \ref{sec4}, we study a clipped perturbation with exact stationarity. Numerical examples supporting the theoretical results are given in  Section \ref{sec5}. Finally the conclusion ends the paper in Section \ref{sec6}.

\section{Setting and a basic estimate}\label{sec2}

Throughout the paper, the following assumption is used.

\medskip
\noindent\textbf{Assumption 1.}
The function $f:\R^n\to\R$ is differentiable, its gradient is
$L$-Lipschitz continuous for some $L>0$, and
$$
   \fmin:=\inf_{x\in\R^n}f(x)>-\infty.
$$
We write
$\Delta_0=f(x_0)-\fmin$.

\medskip
The smoothness assumption gives the standard descent inequality
\begin{equation}\label{eq:descent-lemma}
  f(z)\leq f(x)+\inner{\nabla f(x)}{z-x}
              +\frac{L}{2}\norm{z-x}^2
  \qquad (x,z\in\R^n);
\end{equation}
see, for example, \cite{Bertsekas2016,Nesterov2018}.

For an arbitrary perturbation $y_k$, consider
\begin{equation}\label{eq:generic-update}
   x_{k+1}=x_k-h\nabla f(y_k),
   \qquad 0<h\leq \frac1L.
\end{equation}
The following estimate will be used in all proofs.

\begin{lemma}[Perturbed descent]\label{lem:perturbed-descent}
Let Assumption~1 hold, and let $(x_k)$ be generated by
\eqref{eq:generic-update}.  Then
\begin{equation}\label{eq:perturbed-descent}
  f(x_{k+1})
  \leq f(x_k)
  -\frac{h}{2}\left[
       \norm{\nabla f(x_k)}^2
       -\norm{\nabla f(y_k)-\nabla f(x_k)}^2
     \right].
\end{equation}
Consequently,
\begin{equation}\label{eq:radius-descent}
  f(x_{k+1})
  \leq f(x_k)
  -\frac{h}{2}\left[
       \norm{\nabla f(x_k)}^2-L^2\norm{y_k-x_k}^2
     \right].
\end{equation}
\end{lemma}

\begin{proof}
Put $g_k=\nabla f(x_k)$ and $q_k=\nabla f(y_k)$.
By \eqref{eq:descent-lemma},
\begin{align*}
f(x_{k+1})
&\leq f(x_k)-h\inner{g_k}{q_k}
                   +\frac{Lh^2}{2}\norm{q_k}^2\\
&=f(x_k)-\frac h2\norm{g_k}^2
          -\frac h2(1-Lh)\norm{q_k}^2
          +\frac h2\norm{g_k-q_k}^2.
\end{align*}
Since $h\leq1/L$, the middle term is nonpositive.  This proves
\eqref{eq:perturbed-descent}.  Inequality \eqref{eq:radius-descent} follows
from the Lipschitz continuity of $\nabla f$.
\end{proof}

\subsection{A  sharpness interpretation}

For $r\geq0$, define the local robust increase
\begin{equation}\label{eq:robust-increase}
  {\cal S}_r(x):=
  \max_{\norm{u}\leq r}\bigl(f(x+u)-f(x)\bigr).
\end{equation}
This is one possible finite-radius sharpness measure.  We use it only to
interpret the scale of the perturbation.

\begin{proposition}\label{prop:robust-bound}
Under Assumption~1,
\begin{equation}\label{eq:robust-bound}
  0\leq{\cal S}_r(x)
  \leq r\norm{\nabla f(x)}+\frac L2r^2.
\end{equation}
For the radius $r=\rho\norm{\nabla f(x)}^{1-\alpha}$,
\begin{equation}\label{eq:robust-alpha}
 {\cal S}_r(x)
 \leq
 \rho\norm{\nabla f(x)}^{2-\alpha}
 +\frac{L\rho^2}{2}\norm{\nabla f(x)}^{2(1-\alpha)}.
\end{equation}
\end{proposition}

\begin{proof}
The lower bound follows by choosing $u=0$.  The upper bound is obtained by
applying \eqref{eq:descent-lemma} to $x+u$ and maximizing the right-hand side
over $\norm{u}\leq r$.  Substitution of the stated radius gives
\eqref{eq:robust-alpha}.
\end{proof}

When $\alpha<1$, the right-hand side of \eqref{eq:robust-alpha} vanishes
with the gradient norm.  This observation does not prove that the algorithm
selects a flat minimizer: it only controls the robust increase on the
effective scale used at the current iterate.

\section{The \texorpdfstring{constant-$\rho$}{constant-rho} rule and its
stationarity floor}\label{sec3}

Fix $\rho>0$ and $\alpha\in[0,1]$.  Starting from $x_0$, the
unclipped $\SAMa$ iteration is
\begin{equation}\label{eq:samalpha}
\left\{
\begin{aligned}
 y_k&=x_k+\rho
 \frac{\nabla f(x_k)}{\norm{\nabla f(x_k)}^\alpha},\\
 x_{k+1}&=x_k-h\nabla f(y_k),
\end{aligned}
\right.
\qquad 0<h\leq\frac1L.
\end{equation}
The method stops if $\nabla f(x_k)=0$.  This convention removes the
undefined normalized direction when $\alpha>0$.

Let $s_k=\norm{\nabla f(x_k)}$.  From Lemma~\ref{lem:perturbed-descent},
\begin{equation}\label{eq:fundamental-alpha}
 f(x_{k+1})
 \leq f(x_k)-\frac h2
 \left[s_k^2-L^2\rho^2s_k^{2(1-\alpha)}\right].
\end{equation}
For $0<\alpha\leq1$, define
\begin{equation}\label{eq:floor}
  \varepsilon_\alpha=(L\rho)^{1/\alpha}.
\end{equation}
The bracket in \eqref{eq:fundamental-alpha} is positive exactly when
$s_k>\varepsilon_\alpha$.

\begin{theorem}[Complexity above the floor]\label{thm:epsilon-complexity}
Let Assumption~1 hold, $0<\alpha\leq1$, and let $(x_k)$ be generated by
\eqref{eq:samalpha}.  For every
$\varepsilon>\varepsilon_\alpha$, there exists $k<T$ such that
$s_k\leq\varepsilon$, provided that
\begin{equation}\label{eq:T-epsilon}
 T>
 \frac{2\Delta_0}{
 h\,\varepsilon^{2(1-\alpha)}
 \bigl(\varepsilon^{2\alpha}-L^2\rho^2\bigr)}.
\end{equation}
\end{theorem}

\begin{proof}
Suppose that $s_k>\varepsilon$ for $k=0,\ldots,T-1$.  Then
\begin{align*}
s_k^2-L^2\rho^2s_k^{2(1-\alpha)}
&=s_k^{2(1-\alpha)}
  \bigl(s_k^{2\alpha}-L^2\rho^2\bigr)\\
&>\varepsilon^{2(1-\alpha)}
  \bigl(\varepsilon^{2\alpha}-L^2\rho^2\bigr).
\end{align*}
Summing \eqref{eq:fundamental-alpha} gives
$$
 \Delta_0>
 \frac{hT}{2}\,
 \varepsilon^{2(1-\alpha)}
 \bigl(\varepsilon^{2\alpha}-L^2\rho^2\bigr),
$$
which contradicts \eqref{eq:T-epsilon}.
\end{proof}

The next form gives a direct estimate for the best gradient norm.

\begin{corollary}\label{cor:best-gradient}
Under the assumptions of Theorem~\ref{thm:epsilon-complexity}, let
$$
  b_T=\min_{0\leq k<T}\norm{\nabla f(x_k)}
  \quad\hbox{and}\quad
  c_\alpha=(L\rho)^{2(1-\alpha)/\alpha}.
$$
Then
\begin{equation}\label{eq:best-gradient}
 b_T\leq
 \left(
 L^2\rho^2+\frac{2\Delta_0}{h\,c_\alpha T}
 \right)^{1/(2\alpha)}.
\end{equation}
\end{corollary}

\begin{proof}
If $b_T\leq\varepsilon_\alpha$, the result is immediate.  Otherwise,
$s_k^{2(1-\alpha)}\geq c_\alpha$ for every $k<T$.
Summing \eqref{eq:fundamental-alpha} yields
$$
 Tc_\alpha\bigl(b_T^{2\alpha}-L^2\rho^2\bigr)
 \leq\frac{2\Delta_0}{h},
$$
which is equivalent to \eqref{eq:best-gradient}.
\end{proof}

\begin{remark}\label{rem:alpha-comparison}
The limit in \eqref{eq:best-gradient} is
$(L\rho)^{1/\alpha}$.  If $L\rho<1$, this value decreases when
$\alpha$ decreases.  However, the transient constant contains
$c_\alpha^{-1}$, which may increase at the same time.  There is therefore
no uniform statement that a smaller $\alpha$ always gives a faster method.
Moreover, the physical units of $\rho$ change with $\alpha$, so numerical
comparisons require a fixed scaling of the objective and variables.
\end{remark}

\subsection{The unnormalized endpoint}

The case $\alpha=0$ cannot be obtained by inserting $\alpha=0$ into
\eqref{eq:best-gradient}.  It must be considered separately.  In this case,
$$
  y_k=x_k+\rho\nabla f(x_k),
$$
which is commonly called unnormalized SAM.

\begin{theorem}[Unnormalized SAM]\label{thm:usam}
Let Assumption~1 hold and suppose $L\rho<1$.  For the iteration
\eqref{eq:samalpha} with $\alpha=0$ and $0<h\leq1/L$,
\begin{equation}\label{eq:usam-bound}
 \min_{0\leq k<T}\norm{\nabla f(x_k)}
 \leq
 \left(
 \frac{2\Delta_0}{
 h(1-L^2\rho^2)T}
 \right)^{1/2},
\end{equation}
and
$$
  \norm{\nabla f(x_k)}\longrightarrow0.
$$
\end{theorem}

\begin{proof}
For $\alpha=0$, inequality \eqref{eq:fundamental-alpha} becomes
$$
 f(x_{k+1})\leq f(x_k)
 -\frac h2(1-L^2\rho^2)s_k^2.
$$
Summation gives
$$
 \sum_{k=0}^{\infty}s_k^2
 \leq \frac{2\Delta_0}{h(1-L^2\rho^2)}.
$$
This proves both claims.
\end{proof}

\subsection{Exact behavior on a quadratic}

The upper floor is attained on the simplest smooth strongly convex example.

\begin{theorem}[A tight quadratic example]\label{thm:quadratic}
Let $f(x)=Lx^2/2$ on $\R$, let $h=1/L$, $\rho>0$, and $x_0\neq0$.
For constant-$\rho$ $\SAMa$ with $0<\alpha\leq1$, put
$s_k=|f'(x_k)|$.  Then
\begin{equation}\label{eq:quadratic-recursion}
  s_{k+1}=L\rho\,s_k^{1-\alpha}
\end{equation}
and
\begin{equation}\label{eq:quadratic-limit}
  s_k\longrightarrow(L\rho)^{1/\alpha}.
\end{equation}
  For $\alpha=0$,
$$
   s_{k+1}=L\rho\,s_k.
$$
Thus the gradient converges to zero if $L\rho<1$, stays constant if
$L\rho=1$, and diverges if $L\rho>1$.
\end{theorem}

\begin{proof}
For $x_k\neq0$,
$$
 y_k=x_k+\rho L^{1-\alpha}
       |x_k|^{1-\alpha}\operatorname{sign}(x_k).
$$
Since $h=1/L$, we have $x_{k+1}=x_k-y_k$.  Multiplication by $L$
we obtain \eqref{eq:quadratic-recursion}.
For $\alpha>0$, taking logarithms gives
$$
 \log s_{k+1}=\log(L\rho)+(1-\alpha)\log s_k.
$$
The affine recursion converges to
$\alpha^{-1}\log(L\rho)$, proving \eqref{eq:quadratic-limit}.
The statement for $\alpha=0$ follows directly.
\end{proof}

Theorem~\ref{thm:quadratic} also shows why a comparison with optimally
stepped gradient descent must be made with care.  On the same quadratic,
gradient descent with $h=1/L$ reaches zero in one update, while normalized
SAM does not.  In the terminology used here, this is the constant-radius case
$\alpha=1$ of the constant-$\rho$ rule.

\section{A clipped perturbation with exact stationarity}\label{sec4}

Write
$$
 e_k=\nabla f(y_k)-\nabla f(x_k).
$$
Lemma~\ref{lem:perturbed-descent} gives a genuine decrease whenever
$\norm{e_k}<\norm{\nabla f(x_k)}$.  For the constant-$\rho$
$\SAMa$ rule, Lipschitz continuity gives only
$$
 \frac{\norm{e_k}}{\norm{\nabla f(x_k)}}
 \leq L\rho\,\norm{\nabla f(x_k)}^{-\alpha}.
$$
For $\alpha>0$, this upper bound grows as the gradient becomes small and
reaches one at the scale $(L\rho)^{1/\alpha}$.  This is the mechanism
behind the floor in Section~3.  We remove it by imposing the relative-error
condition
\begin{equation}\label{eq:relative-error}
 \norm{\nabla f(y_k)-\nabla f(x_k)}
 \leq\theta\norm{\nabla f(x_k)},
 \qquad 0<\theta<1.
\end{equation}

Relative-error conditions are standard in inexact-gradient analysis and are
also used in the convergence study of USAM by Khanh et
al.~\cite{Khanh2024}.  The point here is not the inequality itself.  The
following radius rule enforces it for every $\alpha\in[0,1]$, while
retaining the original constant-$\rho$ $\SAMa$ update whenever that
update already satisfies the required scale.
Let $0<\theta<1$, and define
\begin{equation}\label{eq:clipped-rho}
  \rho_k=
  \min\left\{
       \rho,\frac{\theta}{L}
       \norm{\nabla f(x_k)}^\alpha
      \right\}.
\end{equation}
The clipped iteration is
\begin{equation}\label{eq:clipped-sam}
\left\{
\begin{aligned}
 y_k&=x_k+\rho_k
 \frac{\nabla f(x_k)}{\norm{\nabla f(x_k)}^\alpha},\\
 x_{k+1}&=x_k-h\nabla f(y_k),
\end{aligned}
\right.
\qquad 0<h\leq\frac1L.
\end{equation}
Again, the method stops when $\nabla f(x_k)=0$.

For $\alpha>0$, the method is identical to the constant-$\rho$
$\SAMa$ rule whenever
\begin{equation}\label{eq:fixed-region}
  \norm{\nabla f(x_k)}
  \geq\left(\frac{L\rho}{\theta}\right)^{1/\alpha}.
\end{equation}
Below this level, the perturbation becomes
$$
 y_k-x_k=\frac{\theta}{L}\nabla f(x_k).
$$
Thus the two regions have a direct interpretation.  Above the threshold in
\eqref{eq:fixed-region}, the method is exactly the original sharpness-aware
rule.  Below the threshold, the displacement is proportional to the gradient
and satisfies \eqref{eq:relative-error}.  A value of $\theta$ close to one
keeps the constant-$\rho$ behavior closer to stationarity, but gives the
smaller descent factor $1-\theta^2$ below.  A smaller $\theta$ clips
earlier and gives a stronger worst-case descent factor.  This is the tradeoff
between preserving the original perturbation and requiring exact
stationarity.

\begin{theorem}[Convergence of clipped SAM]\label{thm:clipped}
Let Assumption~1 hold, $0\leq\alpha\leq1$, $0<\theta<1$, and
$0<h\leq1/L$.  Let $(x_k)$ be generated by \eqref{eq:clipped-sam}.
Then
\begin{equation}\label{eq:clipped-descent}
 f(x_{k+1})
 \leq f(x_k)
 -\frac h2(1-\theta^2)\norm{\nabla f(x_k)}^2.
\end{equation}
Consequently,
\begin{equation}\label{eq:clipped-rate}
 \min_{0\leq k<T}\norm{\nabla f(x_k)}
 \leq
 \left(
 \frac{2\Delta_0}{h(1-\theta^2)T}
 \right)^{1/2},
\end{equation}
and
\begin{equation}\label{eq:clipped-gradient-limit}
  \norm{\nabla f(x_k)}\longrightarrow0.
\end{equation}
\end{theorem}

\begin{proof}
By \eqref{eq:clipped-rho},
$$
 \norm{y_k-x_k}
 =\rho_k\norm{\nabla f(x_k)}^{1-\alpha}
 \leq\frac{\theta}{L}\norm{\nabla f(x_k)}.
$$
Hence
$$
 \norm{\nabla f(y_k)-\nabla f(x_k)}
 \leq\theta\norm{\nabla f(x_k)}.
$$
Substitution into \eqref{eq:perturbed-descent} gives
\eqref{eq:clipped-descent}.  Summing this inequality yields
$$
 \sum_{k=0}^{\infty}\norm{\nabla f(x_k)}^2
 \leq\frac{2\Delta_0}{h(1-\theta^2)}.
$$
The rate \eqref{eq:clipped-rate} and the limit
\eqref{eq:clipped-gradient-limit} follow.
\end{proof}

\begin{corollary}[Robust increase on the clipped scale]
\label{cor:clipped-sharpness}
Let $r_k=\norm{y_k-x_k}$ in \eqref{eq:clipped-sam}.  Then
\begin{equation}\label{eq:clipped-sharpness}
 {\cal S}_{r_k}(x_k)
 \leq
 \frac{\theta+\theta^2/2}{L}
 \norm{\nabla f(x_k)}^2
 \longrightarrow0.
\end{equation}
\end{corollary}

\begin{proof}
Use $r_k\leq\theta\norm{\nabla f(x_k)}/L$ in
\eqref{eq:robust-bound}, followed by
\eqref{eq:clipped-gradient-limit}.
\end{proof}

\begin{remark}
The clipping rule uses the global Lipschitz constant $L$, which is rarely
known in large-scale applications.  Any known upper bound
$\widehat L\geq L$ can replace $L$ in \eqref{eq:clipped-rho}; the
resulting displacement is more conservative and still implies
\eqref{eq:relative-error}.  A local or adaptive estimate of $L$ is another
possible choice.  For example, a backtracking procedure could increase
$\widehat L_k$ until the directly computable test
$$
 \norm{\nabla f(y_k)-\nabla f(x_k)}
 \leq\theta\norm{\nabla f(x_k)}
$$
is satisfied.  This observation suggests a practical implementation but is
not a convergence theorem for an adaptive method.  The termination of such a
search, the admissible stepsizes, and convergence with local estimates
require a separate analysis and are outside the scope of this paper.
\end{remark}

\section{Numerical illustrations}\label{sec5}

The experiments illustrate the statements proved above; they are not an
exhaustive benchmark.  We compare gradient descent (GD), normalized SAM,
unclipped $\mathrm{SAM}_{1/2}$, and clipped
$\mathrm{SAM}^{\mathrm{clip}}_{1/2}$.  A SAM-type update evaluates
$\nabla f(x_k)$ and $\nabla f(y_k)$, whereas a GD update uses one
gradient.  For this reason, the horizontal axes report gradient evaluations.
The stopping test is based on $\norm{\nabla f(x_k)}$.

The parameters are listed in Table~\ref{tab:parameters}.  The MATLAB
implementation used to produce all figures accompanies the paper.

\begin{table}[H]
\centering
\caption{Parameters of the numerical tests.  The value $\widehat L$ is
used by the clipping rule.}
\label{tab:parameters}
\small
\begin{tabular}{lccccccc}
\toprule
Problem & $h$ & $\rho$ & $\alpha$ & $\theta$ &
$\widehat L$ & maximum updates & tolerance\\
\midrule
Quadratic  & $1$        & $10^{-1}$ & $1,\frac34,\frac12$
           & $0.8$ & $1$    & $60$       & --\\
Rosenbrock & $5\,10^{-4}$ & $10^{-5}$ & $\frac12$
           & $0.8$ & $2000$ & $4\,10^5$ & $10^{-8}$\\
Five-dimensional & $5\,10^{-2}$ & $10^{-3}$ & $\frac12$
           & $0.8$ & $14$   & $10^4$     & $10^{-10}$\\
\bottomrule
\end{tabular}
\end{table}

\subsection{Quadratic floor}

We first take $f(x)=x^2/2$, $x_0=2$, and $h=1$.  The exact floors for
$\alpha=1,3/4,1/2$ are, respectively,
$$
  10^{-1},\qquad 10^{-4/3}\simeq4.64\,10^{-2},
  \qquad 10^{-2}.
$$
Figure~\ref{fig:quadratic} shows that the computed values agree with these
numbers.  With clipping, the gradient norms continue to zero.  On this
example, once clipping is active, the recursion is
$s_{k+1}=\theta s_k$.

\begin{figure}[t]
\centering
\includegraphics[width=0.88\textwidth]{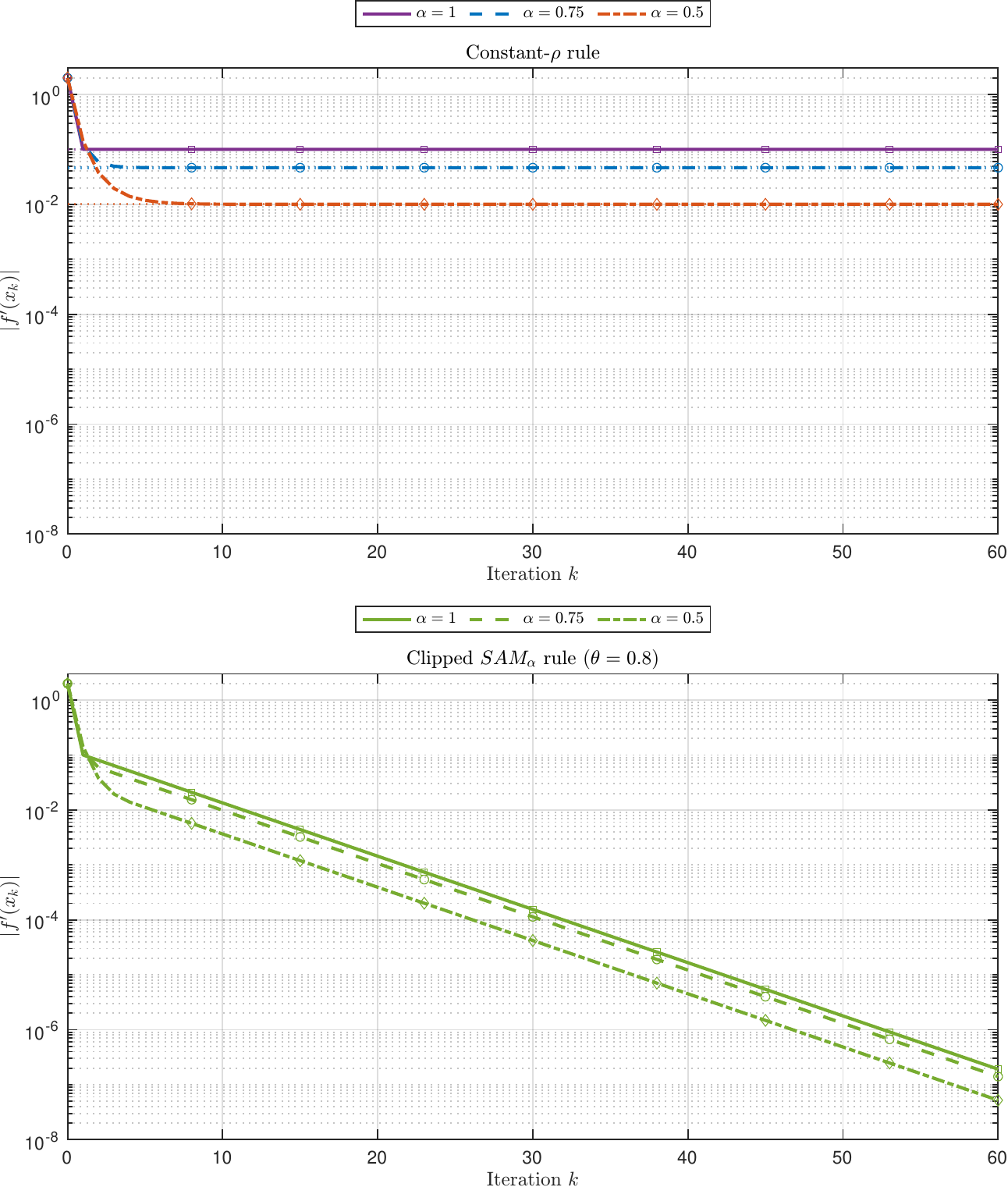}
\caption{Gradient norms for the constant-$\rho$ and clipped rules on
$f(x)=x^2/2$.  The dotted horizontal lines in the upper panel show the exact
floors $\rho^{1/\alpha}$.}
\label{fig:quadratic}
\end{figure}

\subsection{Rosenbrock function}

Consider
\begin{equation}\label{eq:rosenbrock}
 f(x_1,x_2)=100(x_2-x_1^2)^2+(1-x_1)^2,
 \qquad x_0=(-1.2,1).
\end{equation}
The minimum is $f(1,1)=0$.  The function is not globally smooth on
$\R^2$, since its Hessian is unbounded.  This test is therefore an
illustration outside the global assumption of the theorems.  The iterates
remain in the displayed bounded region and $\widehat L=2000$ is used only
for clipping.

The four trajectories are almost indistinguishable at the scale of the level
sets; see Figure~\ref{fig:rosenbrock}.  Their final stationarity is different.
GD reaches the tolerance after $87\,228$ gradient evaluations.  Clipped
$\mathrm{SAM}_{1/2}$ reaches it after $174\,433$ evaluations, close to
twice the GD cost.  Unclipped SAM and unclipped
$\mathrm{SAM}_{1/2}$ stop at the maximum budget with gradient norms
$3.35\,10^{-3}$ and
$1.12\,10^{-5}$, respectively.  The radius plot shows that clipping
becomes active near the end and removes the remaining floor.

\begin{figure}[t]
\centering
\includegraphics[width=0.94\textwidth]{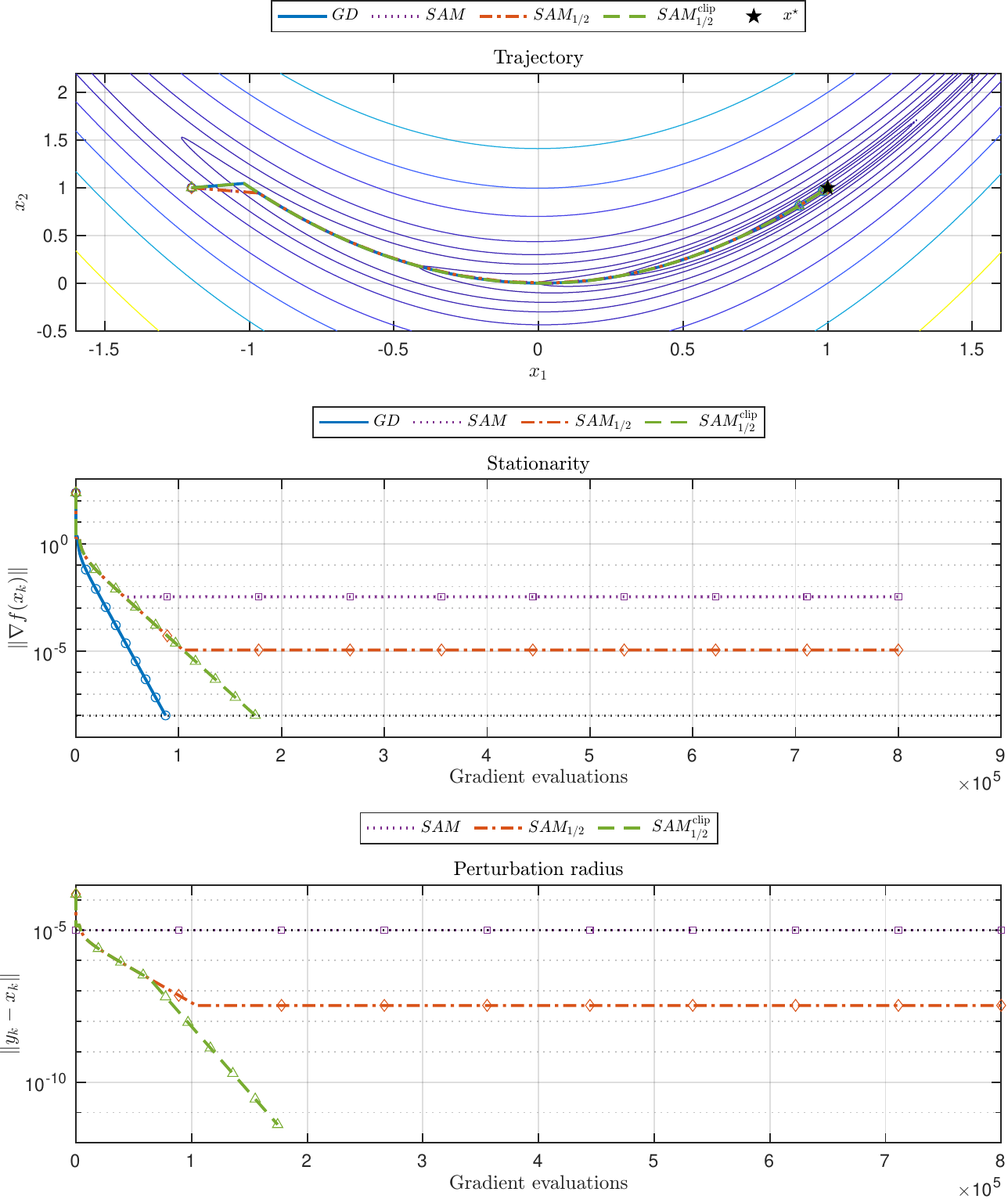}
\caption{Rosenbrock experiment.  The convergence plots use gradient
evaluations on the horizontal axis.}
\label{fig:rosenbrock}
\end{figure}

\subsection{A smooth nonconvex problem in dimension five}

We next consider
\begin{equation}\label{eq:five-dimensional}
 f(x)=\sum_{i=1}^5x_i^2+
 3\sum_{i=1}^5\sin(x_i+x_{i+1}),
 \qquad x_6=x_1,
\end{equation}
from
$$
 x_0=(1.5,-1,0.5,2,-1.5).
$$
The gradient components are
$$
 \frac{\partial f}{\partial x_i}(x)
 =2x_i+3\cos(x_i+x_{i+1})
       +3\cos(x_{i-1}+x_i),
$$
with periodic indices.  The Hessian row sums are bounded by $14$, so
$\nabla f$ is globally $14$-Lipschitz.  The choices $L=14$ and
$h=0.05<1/L$ satisfy the theoretical conditions.

The reference value
$$
 f_{\rm ref}=-12.35796126705
$$
was obtained by 80 deterministic multistart GD runs, using the random seed
one.  It is used only to display an objective gap.  The numerical summary is
given in Table~\ref{tab:results}.  As expected, normalized SAM has the largest
floor, while unclipped $\mathrm{SAM}_{1/2}$ gives a smaller one.  The
clipped method reaches the requested tolerance.  It uses more gradients than
GD, which is consistent with the cost of a SAM update.

\begin{table}[H]
\centering
\caption{Final results.  ``Gradients'' counts all evaluations of
$\nabla f$.  The last column is $f(x_k)$ for Rosenbrock and
$f(x_k)-f_{\rm ref}$ for the five-dimensional problem.}
\label{tab:results}
\small
\begin{tabular}{llrrr}
\toprule
Problem & Method & Updates & Gradients &
$\norm{\nabla f(x_k)}$ \quad Objective/gap\\
\midrule
Rosenbrock
& GD                         & 87\,227 & 87\,228  & $9.999\,10^{-9}\quad 1.252\,10^{-16}$\\
& SAM                        & 400\,000& 800\,001 & $3.346\,10^{-3}\quad 5.588\,10^{-9}$\\
& $\mathrm{SAM}_{1/2}$     & 400\,000& 800\,001 & $1.119\,10^{-5}\quad 6.256\,10^{-14}$\\
& $\mathrm{SAM}^{\rm clip}_{1/2}$
                             & 87\,216 & 174\,433 & $1.000\,10^{-8}\quad 1.252\,10^{-16}$\\
\midrule
Five-dimensional
& GD                         & 153    & 154    & $9.098\,10^{-11}\quad 3.553\,10^{-15}$\\
& SAM                        & 10\,000& 20\,001 & $7.132\,10^{-3}\quad 1.857\,10^{-6}$\\
& $\mathrm{SAM}_{1/2}$     & 10\,000& 20\,001 & $5.083\,10^{-5}\quad 9.434\,10^{-11}$\\
& $\mathrm{SAM}^{\rm clip}_{1/2}$
                             & 136    & 273    & $9.831\,10^{-11}\quad 3.553\,10^{-15}$\\
\bottomrule
\end{tabular}
\end{table}

\begin{figure}[t]
\centering
\includegraphics[width=0.94\textwidth]{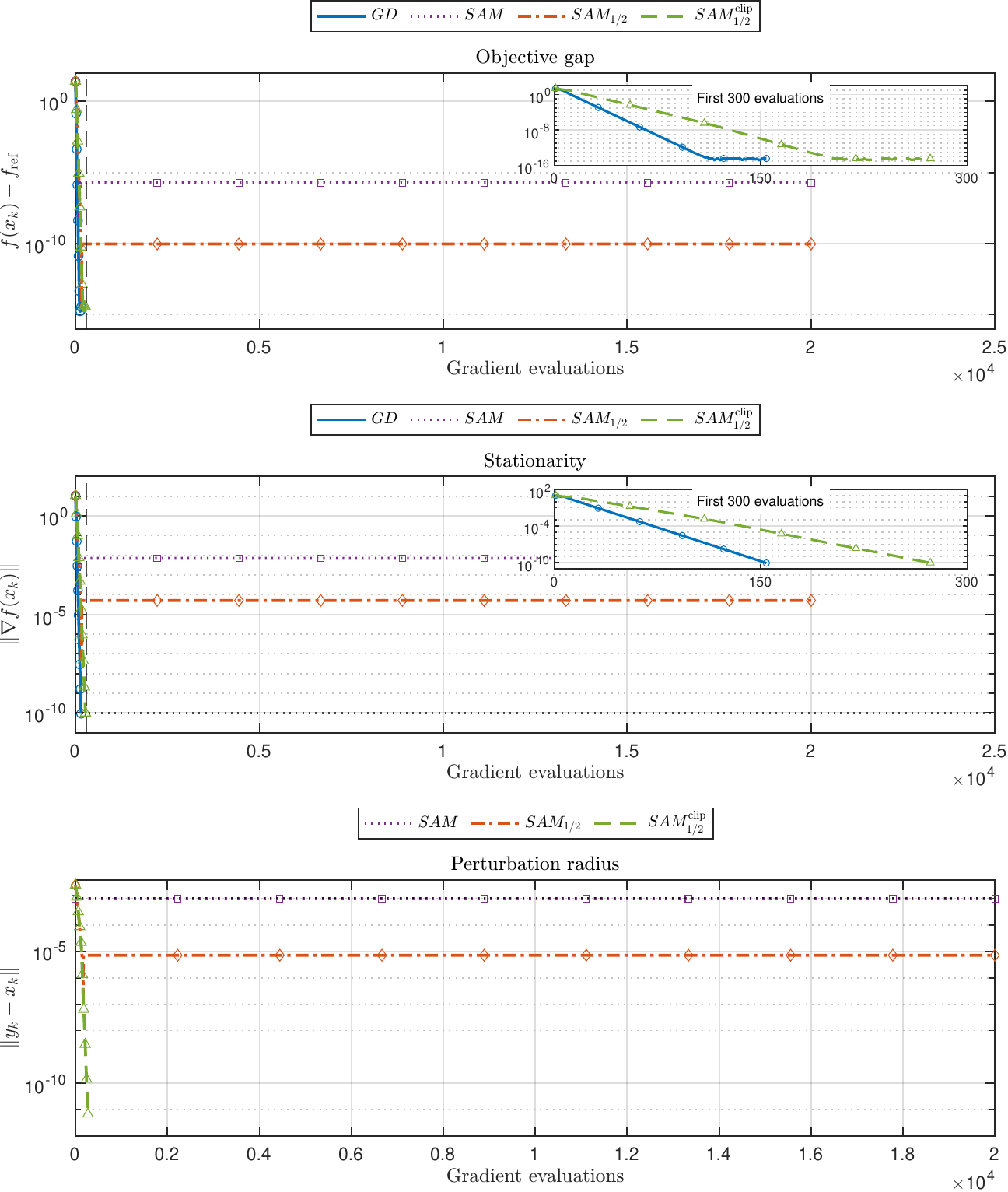}
\caption{Five-dimensional nonconvex problem.  The constant-$\rho$ rules
lead to visible stationarity floors, whereas clipping restores convergence
to the prescribed tolerance.  The insets compare GD and clipped
$\mathrm{SAM}_{1/2}$ over the first 300 gradient evaluations for the
objective gap and stationarity measure; the vertical dashed lines mark the
zoom limit.  GD reaches the tolerance with fewer gradient evaluations.}
\label{fig:five-dimensional}
\end{figure}

\clearpage
\section{Conclusion}\label{sec6}

For the constant-$\rho$ $\SAMa$ rule, the relevant stationarity scale is
$(L\rho)^{1/\alpha}$.  The finite-time estimate and the quadratic example
show that this scale is both an upper guarantee and an attainable limit.
A smaller $\alpha$ may lower the floor when $L\rho<1$, but there is no
uniform improvement because the transient constants also depend on
$\alpha$.  The endpoint $\alpha=0$ behaves differently and is stable
under the condition $L\rho<1$.

The clipped radius preserves the constant-$\rho$ rule outside the final
stationarity region.  Near a stationary point it enforces a relative
gradient error and makes the perturbation proportional to the gradient.  This
gives monotone descent and convergence of the gradient norms to zero.  The
numerical tests agree with this mechanism.  They also show that the clipped
SAM update should not be called faster than gradient descent when gradient
evaluations are counted.

The analysis is deterministic and concerns stationarity only.  It does not
show that the clipped method has the same implicit bias or generalization
behavior as the constant-$\rho$ SAM rule.  Stochastic gradients, a rigorous
line-search or adaptive-$L$ analysis, and the effect of clipping on trained
neural networks remain separate questions.


\end{document}